\documentclass[12pt]{article}
\usepackage{color}
\usepackage{xcolor}
\usepackage{fullpage}
\usepackage{amssymb,amsmath,amsthm}

\usepackage{graphicx}

\usepackage{bbold}

\usepackage{float}

\usepackage{thmtools}

\usepackage{url}

\newtheorem{thm}{Theorem} 
\newtheorem{lemma}[thm]{Lemma}
\newtheorem{cor}[thm]{Corollary} 
\theoremstyle{definition}
\newtheorem{defn}[thm]{Definition}
\newtheorem{claim}[thm]{Claim}

\newcommand{\A}{\mathcal{A}}
\newcommand{\B}{\mathcal{B}}

\newcommand{\F}{\mathcal{F}}

\title{A generalization of diversity for intersecting families}

\author{Van Magnan\thanks{Department of Mathematical Sciences, University of Montana. Email: \texttt{van.magnan@umontana.edu}.
}
\qquad
Cory Palmer\thanks{Department of Mathematical Sciences, University of Montana. Email: \texttt{cory.palmer@umontana.edu}.
Research supported by a grant from the Simons Foundation \#712036.}
\qquad
Ryan Wood\thanks{Department of Mathematical Sciences, University of Montana. Email: \texttt{ryan2.wood@umontana.edu}.}
}

\begin{document}

\maketitle

\begin{abstract}
 Let $\F\subseteq \binom{[n]}{r}$ be an intersecting family of sets and let $\Delta(\F)$ be the maximum degree in $\F$, i.e., the maximum number of edges of $\F$ containing a fixed vertex. The {\it diversity} of $\F$ is defined as $d(\F) := |\F| - \Delta(\F)$. 
 Diversity can be viewed as a measure of distance from the `trivial' maximum-size intersecting family given by the Erd\H os-Ko-Rado Theorem. Indeed, the diversity of this family is $0$. Moreover, the diversity of the largest non-trivial intersecting family \`a la Hilton-Milner is $1$.
 It is known that the maximum possible diversity of an intersecting family $\mathcal{F}\subseteq \binom{[n]}{r}$ is $\binom{n-3}{r-2}$ as long as $n$ is large enough. 

 We introduce a generalization called the {\it $C$-weighted diversity} of $\F$ as $d_C(\F) := |\F| - C \cdot \Delta(\F)$. We determine the maximum value of $d_C(\F)$ for intersecting families $\mathcal{F} \subseteq \binom{[n]}{r}$ and characterize the maximal families for $C\in \left[0,\frac{7}{3}\right)$ as well as give general bounds for all $C$. Our results imply, for large $n$, a recent conjecture of Frankl and Wang concerning a related diversity-like measure.
Our primary technique is a variant of Frankl's Delta-system method. 
\end{abstract}

\section{Introduction}

Let $[n]:=\{1,2,\ldots, n\}$ denote the standard $n$-element set. Let $2^{[n]}$ denote the collection of all subsets of $[n]$ and $\binom{[n]}{r}$ the collection of all $r$-subsets of $[n]$. A subset $\F \subseteq 2^{[n]}$ is a {\it family} and a subset $\F \subseteq \binom{[n]}{r}$ is an {\it $r$-uniform family}. 
We will frequently use the terminology of hypergraphs and use the term {\it edge} for a member $F$ of a family $\F$ and {\it vertex} for an element of the {\it ground set} $[n]$. 

A family $\F$ is \emph{intersecting} if every pair of edges in $\F$ has a non-empty intersection.
The celebrated Erd\H os-Ko-Rado Theorem~\cite{EKR} establishes the maximum size of an intersecting family $\F \subseteq \binom{[n]}{r}$ as $\binom{n-1}{r-1}$ for $n \geq 2r$ with equality only if $\F$ is a \emph{star}, i.e, the family of all $r$-sets containing a fixed vertex $x$, when $n>2r$. Many generalizations and extensions of this fundamental theorem have been introduced and studied. A survey of Frankl and Tokushige~\cite{FrTo} gives an excellent history and overview of these problems.
A classical extension extension is to investigate the \emph{stability} of the EKR star construction. If we view the star as a ``trivial'' solution to the intersecting problem, then a natural question is to ask for the largest intersecting family not contained in a star. Hilton and Milner~\cite{HM} answered this question and showed the maximum is attained by the family of size $\binom{n-1}{r-1} - \binom{n-r-1}{n-1} + 1$
consisting of an edge $Y$ and all edges that contain a fixed vertex $x \not \in Y$ and at least one vertex of $Y$. For more results of this type, maximizing the size of an intersecting family without being a subfamily of one of a given set of families, we direct the reader towards work by Han and Kohayakawa~\cite{HaKo} and later by Kostochka and Mubayi~\cite{KoMu}.

We may investigate the size of intersecting families that are non-trivial by maximizing related statistics that are known to be small for star constructions. The {\it maximum degree} $\Delta(\F)$ of a family $\F$ is the size of the largest star subfamily of $\F$, i.e., the maximum number of edges containing a fixed vertex.
The {\it diversity} $d(\F)$ of a family $\F$ is the difference between its size $|\F|$ and maximum degree $\Delta(\F)$, i.e.,
\[
d(\F) := |\F| - \Delta(\F).
\]

One may observe that the diversity of a star (and its sub-families) is $0$. Lemons and Palmer \cite{LePa} (answering a question of Katona) proved that an intersecting family $\F \subseteq \binom{[n]}{r}$ has diversity
\[
d(\F) \leq \binom{n-3}{r-2}
\]
whenever $n \geq 6r^3$.
Improvements on the threshold on $n$ were given a series of papers (see \cite{Fr-B, Fr-A, ku-thresh}). The current best threshold is  $n \geq 36r$ is given by Frankl and Wang~\cite{frankl2023improved}. Frankl~\cite{Fr-A} conjectured a best-possible threshold on $n$, but this was disproved by Huang~\cite{huang},  Kupavskii \cite{ku-thresh} and, via linear programming, Wagner~\cite{wagner}.

For $n$ large enough, an intersecting family $\F \subseteq \binom{[n]}{r}$ attaining the maximum diversity necessarily contains (up to isomorphism) 
the family of all edges that intersect $[3]=\{1,2,3\}$ in exactly two vertices. We denote this family by
\[
\A_\mathbb{K} := \left\{E\in \binom{[n]}{r} : |E\cap [3]| = 2\right\}.
\]
Note that $\F$ may also include any number of edges that contain $\{1,2,3\}$,
that is, $\F$ is contained in the family denoted by 
\[
\A_{\mathbb{K}^+} := \left\{E\in \binom{[n]}{r} : |E\cap [3]| \ge 2\right\}.
\]
 We often refer to any family $\F$ satisfying $\A_{\mathbb{K}}\subseteq \F \subseteq \A_{\mathbb{K}^+}$ as a \emph{``two out of three'' family}. Note that edges of $\A_{\mathbb{K}^+} \setminus \A_\mathbb{K}$ make a contribution of $0$ to the diversity of a ``two out of three'' family $\F$ as each such edge increases both $|\F|$ and $\Delta(\F)$ by $1$.

In this paper we introduce the following generalization of diversity.
\begin{defn}[Diversity]
The \emph{$C$-weighted
diversity} of a family $\F$ is
\[
d_C(\F) := |\F| - C \cdot \Delta(\F).
\]    
\end{defn}

We are concerned with the maximum value of $d_C(\F)$ over intersecting families $\F \subseteq \binom{[n]}{r}$.
For $C=0$, this is simply the maximum size of an intersecting family $\F \subseteq \binom{[n]} {r}$, so this question can be viewed as a generalization of the Erd\H os-Ko-Rado Theorem. 

In fact, for $0 \leq C <1$, it is easy to see that for large enough $n$,
\begin{equation}\label{proto-bound}
d_C(\F)\leq (1-C) \binom{n-1}{r-1}    
\end{equation}
with equality only if $\F$ is a star. Indeed, if $\F$ is not a star, then by the Hilton-Milner Theorem, $d_C(\F) < |\F| \leq \binom{n-1}{r-1} - \binom{n-r-1}{n-1} + 1 \leq r \binom{n-2}{r-2}<(1-C) \binom{n-1}{r-1}$ for $n$ large enough. 

The introduction of the $C$-weighted diversity parameter has several motivations. First are  recent alternative generalizations of diversity (see \cite{Fr-A, frankl2022best, frankl2023improved}). Another are the degree versions (see \cite{dezafrankl} for a summary) of the Erd\H os-Ko-Rado Theorem that determine the maximum of $|\F|$ when $\Delta(\F) \leq c |\F|$ for a constant $c$. Our results are closely related to these and we obtain similar extremal constructions.
The final motivation is a problem of Kupavskii and Zakharov~\cite{KuZa} that asks to find the largest possible value of $c>1$ for which
\[
c \cdot d(\F) + \Delta(\F) = c|\F| - (c-1) \Delta(\F) \leq \binom{n-1}{r-1},
\]
holds for intersecting families $\F \subseteq \binom{[n]}{r}$ with $n>2r$. Solving this is essentially equivalent to establishing inequality~(\ref{proto-bound}) for all $n > 2r$.

When $C=1$, the $C$-weighted diversity $d_C(\F)$ is the ordinary diversity $d(\F)$. When maximizing $d_C(\F)$ in this case (and for smaller $C$), one can assume that a family is saturated, i.e., the addition of any edge will violate the intersecting condition. This assumption is often useful in similar problems. However, for larger $C$, we cannot make this assumption. Indeed, removing an edge containing all vertices of maximum degree increases the diversity. 

Our first result is an extension of the original diversity bound to the range $C \in \left[1,\frac{3}{2}\right)$.

\begin{restatable}{thm}{mainprop}\label{32-ker-prop}
Fix $r\geq 3$ and let $\F \subseteq \binom{[n]}{r}$ be an intersecting family of maximum $C$-weighted diversity $d_C(\F)$. For $n$ large enough,
\begin{align*}
       \A_{\mathbb{K}} &\subseteq \F \subseteq \A_{\mathbb{K}^+}& \text{ for } C=1,\\
    \F &= \A_{\mathbb{K}} & \text{ for }1< C<\frac{3}{2}.
\end{align*}    
Further, for $1\leq C < \frac{3}{2}$, 
\[
d_C(\F) \leq (3-2C)\binom{n-3}{r-2}.
\]
\end{restatable}

For $C\geq \frac{3}{2}$ we have $d_C(\A_\mathbb{K})\leq 0$. In this case, a different construction is required for positive $C$-weighted diversity. Recall that a ($n,r,\lambda$)-design is a family $\F \subseteq \binom{[n]}{r}$ such that each pair of distinct vertices is found in exactly $\lambda$ edges. 
The {\it Fano plane} is the unique (up to isomorphism) $(7,3,1)$-design. It is given by the following $7$ edges on vertex set $[7]=\{1,2,3,\ldots,7\}$:
\[
\mathbb{F} := \{124,137,156,235,267,346,457\}.
\]
 Let $\mathbb{F}^+ := \mathbb{F}\cup \left(\binom{[7]}{4}\setminus \mathbb{F}^c \right)$ denote the union of $\mathbb{F}$ with the 4-sets of $[7]$ that have nonempty intersection with the edges of the Fano plane. Now, define
 \[
\A_\mathbb{F} := \left\{E\in \binom{[n]}{r} :  E\cap [7] \in \mathbb{F}\right\}, \
\text{and}
\]
\[
\A_{\mathbb{F}^+} := \left\{E\in \binom{[n]}{r} : E\cap [7] \in \mathbb{F}^+\right\}.
\]

 Our main theorem characterizes the structure of families with maximum diversity $d_C(\F)$ in the range $C \in \left[\frac{3}{2}, \frac{7}{3}\right)$.

\begin{restatable}{thm}{mainthm}\label{main-gen}
Fix $r\geq 3$ and let $\F \subseteq \binom{[n]}{r}$ be an intersecting family of maximum $C$-weighted diversity $d_C(\F)$. For $n$ large enough,
\begin{align*}
   \F &= \A_{\mathbb{F}^+} &\text{ for }\frac{3}{2}\leq C<\frac{7}{4},\\
       \A_{\mathbb{F}} &\subseteq \F \subseteq \A_{\mathbb{F}^+}\text{ and } d_C(\F\setminus \A_{\mathbb{F}})=0 & \text{ for } C=\frac{7}{4},\\
    \F &= \A_{\mathbb{F}} & \text{ for }\frac{7}{4}< C<\frac{7}{3}.
\end{align*}
\end{restatable}

Observe that the above result is only concerned with $\frac{3}{2} \leq C < \frac{7}{3}$ as $d_C(\A_\mathbb{F})\le 0$ for $C\geq \frac{7}{3}$. 
When $C=\frac{7}{4}$, for an extremal family $\F$, an averaging argument will show that the edges of $\F\setminus \A_{\mathbb{F}}$ have contribution to $C$-weighted diversity at most zero. For this contribution to be non-negative, each vertex in the underlying Fano plane must be contained in the same number of edges of $\F$. As this holds for $\A_\mathbb{F}$, the same must hold for $\F\setminus \A_{\mathbb{F}}$. Observe that this parallels the case of Theorem~\ref{32-ker-prop}  with $C=1$, in which extremal families are those containing $\A_\mathbb{K}$ that may have additional edges containing all three specially identified vertices. Similar to that case, we may ignore these edges when computing $d_C(\F)$.

So, by computing $d_C(\A_{\mathbb{F}^+})$ and $d_C(\A_{\mathbb{F}})$ we get the following corollary.

\begin{cor}\label{cor-bound}
Let $\F \subseteq \binom{[n]}{r}$ be an intersecting family. For $n$ large enough,
\[
d_C(\F) \leq 
\begin{cases}
   \displaystyle (7-3C) \binom{n-7}{r-3} + (28-16C)\binom{n-7}{r-4} &\text{ for }\frac{3}{2}\leq C<\frac{7}{4},\\
    \displaystyle  (7-3C) \binom{n-7}{r-3} & \text{ for }\frac{7}{4}\leq C<\frac{7}{3},
\end{cases}
\]
with equality exactly when $\F$ is isomorphic to $\A_{\mathbb{F}^+}$ in the former case and $\A_{\mathbb{F}}$ in the latter (with $C \neq \frac{7}{4}$).    
\end{cor}

Note that $d_C(\F) = \Theta(n^{r-1})$ for $0\le C<1$, $d_C(\F) = \Theta(n^{r-2})$ for $1\le C<\frac{3}{2}$, and $d_C(\F) = \Theta(n^{r-3})$ for $\frac{3}{2}\le C<\frac{7}{3}$. 
In the concluding remarks (Section~\ref{final-section}),
we sketch an argument similar to the proof of Theorem~\ref{main-gen} that shows $d_C(\F)=O(n^{r-4})$ when $C\geq \frac{7}{3}$, and we provide a construction yielding $d_C(\F) = \Theta(n^{r-4})$ for $\frac{7}{3}\le C<\frac{13}{4}$. This means that Corollary~\ref{cor-bound} 
is best-possible in the sense that the bounds cannot be extended to further values of $C$.

In general we have the following bounds for all $C$.

\begin{restatable}{thm}{thmthree}\label{main-gen-C}
Fix $r\geq 3$ and let $\F \subseteq \binom{[n]}{r}$ be an intersecting family of maximum $C$-weighted diversity $d_C(\F)$.
If $q<r$ is a positive integer with $\frac{q^2+q+1}{q+1} \leq C$, then
\[
d_C(\F) =  O\left(n^{r-q-1}\right).
\]
Furthermore, if there exists a prime power $p<r$ with $C< \frac{p^2+p+1}{p+1}$, then
\[
d_C(\F)=\Omega\left(n^{r-p-1}\right).
\]

\end{restatable}

In the next section we introduce our main tool in this paper---a variant of Frankl's Delta-system method we call the \emph{flower base}. In general, these tools force the threshold on $n$ in the theorems to be superexponential in $r$. It seems likely that the correct threshold is polynomial in $r$ as is the case for ordinary diversity.

In Section~\ref{main-section} we prove Theorems~\ref{32-ker-prop}, \ref{main-gen} and \ref{main-gen-C}. In Section~\ref{stability-section}, we prove a stability theorem for ordinary diversity $d(\F)$ and give a new short proof of a stability theorem of Kupavskii and Zakharov~\cite{KuZa} (see also Frankl~\cite{Fr-ekr-deg}) for large $n$.

\section{Preliminaries}

For a family $\F \subseteq 2^{[n]}$ and set
 $S \subset [n]$, the \emph{link} $\F(S)$ is the family
\[
\F(S) := \{F \setminus S \, : \, S \subset F \in \F\}.
\]
The \emph{unlink} $\F(\overline{S})$ is the family
\[
\F(\overline{S}) := \{F\in \F \, : \, F\cap S = \emptyset\}.
\]
Occasionally, it will be allowed that $S = \emptyset$ in which case $\F(S) = \F$. When $S=\{x\}$ is a single element, we write $\F(x)$ and $\F(\overline{x})$ instead of $\F(\{x\})$ and $\F(\overline{\{x\}})$, respectively. For a $2$-vertex set $\{x,y\}$ we also write $\F(xy)$ for $\F({x,y})$.
It is often convenient to partition a family $\F$ by a vertex $x$, i.e.,
\[
|\F| = |\F(x)| + |\F(\overline{x})|.
\]

A \emph{transversal} of a family $\F$ is a set of vertices $T$ such that every edge of $\F$ has a non-empty intersection with $T$. Note that in an intersecting family each edge is necessarily a transversal of the family. The size of a minimum transversal of $\F$ is denoted $\tau(\F)$.

Recall that a pair of families $\A, \B$ are {\it cross-intersecting} if $A \cap B \neq \emptyset$ for all $A \in \A$ and $B \in \B$. A theorem of Hilton and Milner~\cite{HM} (see Simpson~\cite{simp} for a streamlined proof), implies that if $\A,\B \subseteq \binom{[n]}{r}$ are non-empty cross-intersecting families with $n \geq 2r$, then
    \[
    |\A| + |\B| \leq \binom{n}{r} - \binom{n-r}{r} + 1.
    \]
Therefore, if we allow $\A$ or $\B$ to be empty we have
\begin{equation}\label{cross-ineq}
|\A| + |\B| \leq \binom{n}{r}.    
\end{equation}

\subsection{Flower base}

We use the following notion defined in~\cite{jukna}.

\begin{defn}[Flower]
Fix $\alpha \geq 1$.  A \emph{flower with threshold $\alpha$ and core} $Y$ is a family $\mathcal{S} =\{S_1,\dots, S_m\}$ such that $Y=\bigcap_{S\in \mathcal{S}}S$ and the family of \emph{petals} $\mathcal{S}(Y) =\{ S_1 \setminus Y, \dots, S_m \setminus Y\}$ has minimum transversal of size greater than $\alpha$, i.e, $\tau(\mathcal{S}(Y)) > \alpha$.   
\end{defn}

\begin{figure}[H]
\begin{center}
{\includegraphics[scale=1.2]{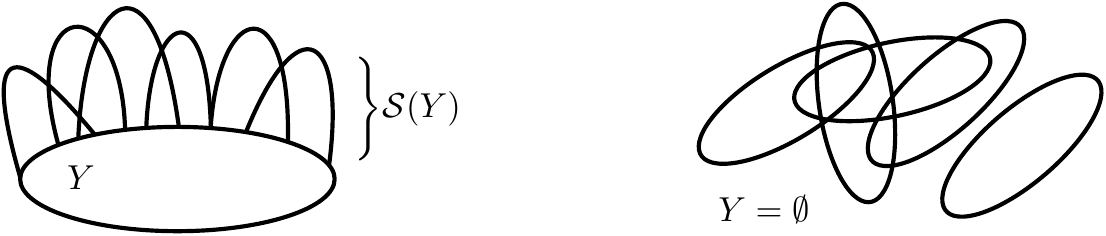}}
\end{center}
\caption{Left, a flower $\mathcal{S}$ with threshold $\alpha$ and core $Y$ has $\tau(\mathcal{S}(Y)) > \alpha$, and right, a flower with core $Y=\emptyset$ and petals of minimum transversal size $3$.}
\end{figure}

A flower is closely related to the better-known \emph{sunflower} which has the requirement that the petals be pairwise disjoint. Indeed, a sunflower with more than $\alpha$ petals and core $Y$ will be a flower with threshold $\alpha$ and core $Y$. The celebrated Erd\H os-Rado sunflower lemma~\cite{sunflower} establishes a Ramsey-like result for the existence of sunflowers in large families. A very similar proof (observed in \cite{HJP}) can be used to prove a counterpart for flowers. A key difference is that the optimal threshold for sunflowers is a famous open problem (see e.g. \cite{alweiss} for a recent breakthrough), but the Erd\H os-Rado argument gives a sharp threshold for flowers with integral threshold. We include a proof here, and sharp construction below, for the sake of completeness.

\begin{lemma}[Flower Lemma \cite{HJP}]\label{flower-lemma}
Fix $\alpha \geq 1$. Let $\F$ be a family of $r$-sets. If $|\F| > \alpha^r$, then $\F$ contains a flower with threshold $\alpha$.
\end{lemma}

\begin{proof} We proceed by induction on $r$. When $r=1$, we have a family of more than $\alpha$ singletons sets. They are clearly disjoint, so the family itself forms a flower with core $\emptyset$ and threshold $\alpha$. So let $r>1$ and assume the theorem holds for smaller values.

Let $T$ be a minimum transversal of $\F$. If $T>\alpha$, the family $\F$ is again a flower with core $\emptyset$ and threshold $\alpha$. Assume otherwise: now, $T$ is a transversal of $\F$, so it  intersects each member of $\F$. Applying pigeonhole, there exists an $x\in T$ such that the number of edges of $\F$ containing $x$ is at least
\[
\frac{|\F|}{|T|}> \frac{\alpha^r}{\alpha} = \alpha^{r-1}.
\]
Consider now the link $\F(x)$, which is a family of $(r-1)$-sets and contains more than $\alpha^{r-1}$ edges. By induction, it contains a flower of core $Y$ and threshold $\alpha$. Adding $x$ to each edge of this flower gives a flower with core $Y\cap \{x\}$ and threshold $\alpha$ in the original family, as needed.
\end{proof}

We exhibit sharpness of the bound for integer $\alpha$ as follows. Let $\mathcal{M}\subseteq \binom{[\alpha r]}{\alpha}$ be a set of $r$ pairwise vertex-disjoint sets each of size $\alpha$. Define $\F$ as the family of minimal transversals of $\mathcal{M}$. One can easily see that $|\F|=\alpha^r$. Let $Y$ be the core of any flower in $\F$. Observe that $|Y|<r$, so the core of this flower $Y$ is disjoint from some member $M \in \mathcal{M}$. But $M$ is of size $\alpha$ and is a transversal of $\F$, so it must intersect each petal of the flower with core $Y$, i.e., $\F$ has no flower with threshold $\alpha$.

We now define our main tool for this paper. It is a variant of the Delta-system method developed by Frankl (see \cite{frankl-delta}) that incorporates transversal analysis. The Delta-system method leverages sunflowers. Instead, leveraging flowers tends to provide an improved threshold on $n$ (although still typically superexponential). We employ notation in this paper that is reminiscent to that used for Delta-systems (see \cite{KoMu} for a recent application).

\begin{defn}[Flower base]
Let $\F \subseteq \binom{[n]}{r}$ and fix $\alpha \geq \tau(\F)$. 
Define $\F^*$ to be the family of subsets non-empty  $Y \subset [n]$ such that $Y$ is the core of a flower of threshold $\alpha$ in $\F$. The \emph{flower base (of threshold $\alpha$)} is the family $\mathfrak{B}\F$ of inclusion-minimal members of $\F^* \cup \F$.
\end{defn}

Typically we are interested in the flower base with threshold $\alpha$ a function of $r$. For example, if $\F$ is intersecting, then often we set $\alpha : = r \geq \tau(\F)$. In a forthcoming manuscript~\cite{MPW} we show further applications of this method that arise from selecting different values of $\alpha$.

The goal of the flower base is to provide a minimal encoding of a family.  However, some edges of the family may not be a member of a flower, so taking only the inclusion-minimal members of $\F^*$ may ignore some members of $\F$. We illustrate this as follows.
Let $r\ge 3$ and $n\ge 2r$ and put $[2,r+1]:=\{2,3,\dots, r+1\}$. Then the Hilton-Milner family

\[
\F := \{
F \in \binom{[n]}{r} : 1 \in F, F \cap [2,r+1] \neq \emptyset
\}
\cup
[2,r+1],
\]
has flower base 
\[
\mathfrak{B}\F =
\{
\{1, j\} : 2 \le j \le r+1
\}
\cup
[2,r+1].
\]

Note that the set $[2,r+1]$ appears in the flower base by virtue of taking inclusion-minimal elements from $\F^* \cup \F$ instead of $\F^*$ alone.

Observe that a member in the flower base $\mathfrak{B}\F$ is either an edge of $\F$ or a set $Y$ that is the core of a flower formed by edges of $\F$.
Therefore, every edge of $\F$ contains some member of $\mathfrak{B}\F$, so the following frequently useful bound on $|\F|$ is immediate:

\begin{equation}\label{b-bound}
    |\F|  \leq \sum_{B \in \mathfrak{B}\F} \binom{n-|B|}{r-|B|}.
\end{equation}

Much of the versatility of the flower base comes from the inheritance of many useful properties of the original family (e.g.\ intersecting). 

\begin{lemma}[Inheritance Lemma]\label{inherit}
Let $\F \subseteq \binom{[n]}{r}$ be a family with flower base $\mathfrak{B}\F$ of threshold $\alpha \geq r$. 

\begin{enumerate}
 \item[(1)] $\mathfrak{B}\F$ is Sperner, i.e., no member of $\mathfrak{B}\F$ contains another.
 \item[(2)] If $\F$ is intersecting, then $\mathfrak{B}\F$ is intersecting. Moreover, each $B \in \mathfrak{B}\F$ is a transversal of $\F$.
\item[(3)] The minimum transversal size of $\mathfrak{B}\F$ satisfies $\tau(\mathfrak{B}\F) = \tau(\F)$.
\end{enumerate}

\end{lemma}

\begin{proof}
 The flower base $\mathfrak{B}\F$ consists of inclusion-minimal sets, so it is Sperner by definition.  In order to prove (2), suppose that $\mathfrak{B}\F$ contains disjoint sets $B_1,B_2$. As $\F$ is intersecting we may assume at least one of $B_1,B_2$, say $B_1$, is not an edge of $\F$, so $B_1$ is the core of a flower of threshold $\alpha$.
 There are at most $r \leq \alpha$ vertices in $B_2$, so $B_2$ is not a transversal of the petals on $B_1$. Therefore, there is an $F \in \F$ that contains $B_1$ and is disjoint from $B_2$.
 If $B_2 \in \F$, then $B_2$ and $F$ are disjoint edges of $\F$, a contradiction. Therefore, $B_2$ is the core of a flower with threshold $\alpha$. Again, there are at most $r \leq \alpha$ vertices in $F$, so $F$ is not a transversal of the petals on $B_2$. Therefore, there is an $F' \in \F$ that contains $B_2$ and is disjoint from $F$, a contradiction.

Now fix $B \in \mathfrak{B}\F$. The petals on $B$ have minimum transversal greater than $\alpha \geq r$, so no edge of $\F$ can be a transversal of these petals. Therefore, as $\F$ is intersecting, each edge intersects $B$, i.e., $B$ is a transversal of $\F$.
 
For (3), simply observe that since a transversal of $\mathfrak{B}\F$ is a transversal of the family $\F$ we have $\tau(\mathfrak{B}\F) \geq \tau(\F)$. On the other hand, observe that a minimum transversal of $\F$ must intersect each $B \in \mathfrak{B}\F$ as $B$ is either an edge of $\F$ or the core of a flower whose petals have transversal greater than $\alpha \geq r \geq \tau(\F)$.
\end{proof}

The following lemma is at the heart of establishing bounds via the flower base. With it, we have that the flower base is of constant size (as long as the minimum transversal size is constant, e.g.\ when the family is intersecting). The proof is an adaptation of a proof  by Frankl~\cite{frankl-delta} of a corresponding result. 
Let $\mathfrak{B}_i\F$ be the members of $\mathfrak{B}\F$ of cardinality $i$, i.e.,
\[
\mathfrak{B}_i\F:=\{B\in \mathfrak{B}\F:|B|=i\}.
\]

\begin{lemma}\label{size-lemma}
Let $\F \subseteq \binom{[n]}{r}$ be a family.
If $\mathfrak{B}\F$ is a flower base with threshold $\alpha \geq \tau(\F)$, then
\[
|\mathfrak{B}\F| \leq  r\alpha^{r}.
\]
\end{lemma}

\begin{proof}
 Suppose for the sake of a contradiction that $|\mathfrak{B}\F| > r\alpha^{r}$. Then $|\mathfrak{B}_i\F| > 
 \alpha^{r} \geq \alpha^i$ for some $1 \leq  i \leq r$. 
 Applying the Flower Lemma (Lemma~\ref{flower-lemma}) to $\mathfrak{B}_i\F$ gives a flower $\mathcal{S} \subseteq \mathfrak{B}_i\F$ with threshold $\alpha$ and core $Y$, i.e, $\tau(\mathcal{S}(Y)) > \alpha$.

Let $\mathcal{S}'$ be the edges of $\F$ that contain a member of the flower $\mathcal{S}$. Each edge of $\mathcal{S}'$ necessarily contains $Y$. 
Let $T$ be a minimum transversal of $\mathcal{S}'(Y)$. 
If $T$ is also a transversal of $\mathcal{S}(Y)$, then $|T|>\alpha$. Otherwise, there is a $B \in \mathcal{S}$ that is disjoint from $T$. 
The edges of $\F$ that contain $B$ are in $\mathcal{S}'$, so then
 $T$ must be a transversal of the petals on $B$, i.e., $|T|>\alpha$. In either case, this implies that $\mathcal{S'}$ is a flower with core $Y$ and threshold at least $\alpha$. 
Note that $Y \neq \emptyset$, as otherwise $\tau(\F) \geq \tau(\mathcal{S'}) > \alpha$, a contradiction.
Therefore, $Y \in \mathfrak{B}\F$. However, $Y \subsetneq B \in \mathfrak{B}_i\F$, so this violates the Sperner property of the flower base $\mathfrak{B}\F$.
\end{proof}

Throughout this paper, $\F \subseteq \binom{[n]}{r}$ is intersecting, thus $r \geq \tau(\F)$. So we will always use the flower base $\mathfrak{B}\F$ with threshold $\alpha = r$. For ease of notation, denote by $r_0 := r^{r+1}$ the bound given by Lemma~\ref{size-lemma}. Combining (\ref{b-bound}) and Lemma~\ref{size-lemma} we have the following bound on $|\F|$ for any $1 \leq k \leq r$,
\begin{equation}\label{combo-bound}
|\F| \leq \sum_{i=1}^{k-1} |\mathfrak{B}_i\F| \binom{n-i}{r-i} + r_0\binom{n-k}{r-k}.
\end{equation}

\section{Generalized diversity}\label{main-section}

Observe the immediate and often useful bound on diversity for all non-empty $S \subset [n]$:
\begin{equation}\label{link-bound}
d_C(\F) \leq  |\F| - C  |\F(S)|.    
\end{equation}

We also use the following lemma.

\begin{lemma}\label{ker-lemma}
     If $\F'\subseteq \F$ has a vertex of degree at least $\frac{1}{C}|\F'|$, then  
     \[
     d_C(\F) \leq |\F \setminus \F'|.
     \]
     In particular, if $\F$ has a vertex of degree at least $\frac{1}{C}|\F|$, then $d_C(\F) \leq 0$.
\end{lemma}

\begin{proof}
Immediately,
\[
d_C(\F) =|\F| - C \cdot  \Delta(\F) \leq  |\F \setminus \F'| + |\F'| -  C  \frac{1}{C} |\F'| =  |\F \setminus \F'|. \qedhere
\]
\end{proof}

A common use of Lemma~\ref{ker-lemma} is to conclude that $d_C(\F) \leq 0$ when $C\geq \frac{3}{2}$ and $\F$ has a vertex in $\frac{2}{3}$ of its edges.

We now prove the first extension of the original diversity bound (for large $n$).

\mainprop*

\begin{proof}
Let $\F \subseteq \binom{[n]}{r}$ be an intersecting family of maximum $C$-weighted diversity. We may assume $d_C(\F) \geq d_C(\A_\mathbb{K}) = (3-2C) \binom{n-3}{r-2}$.
Let $\mathfrak{B}\F$ be the flower base of threshold $r$.
If $\mathfrak{B}\F$ contains a singleton set $\{x\}$, then by Lemma~\ref{inherit}, 
every edge of $\F$ contains $x$, i.e., $\F$ has diversity at most $0$. If $\mathfrak{B}\F$ has no member of size $2$, then (\ref{combo-bound}) gives
\[
d_C(\F) <|\F| \leq r_0 \binom{n-3}{r-3} < (3-2C)\binom{n-3}{r-2} = d_C(\mathcal{A}_\mathbb{K}),
\]
for $n$ large enough, a contradiction.
Therefore, $\mathfrak{B}\F$ contains a member of size $2$, i.e., $\mathfrak{B}_2\F \neq \emptyset$. Suppose that the members $\mathfrak{B}_2\F$ have a common element $x$. Then (\ref{combo-bound}) gives
\begin{align*}
d_C(\F) & \leq |\F| - C  |\F(x)| \\
&\leq |\F(x)| +  r_0 \binom{n-3}{r-3} - C  |\F(x)| \leq r_0 \binom{n-3}{r-3}  < (3-2C)\binom{n-3}{r-2} = d_C(\mathcal{A}_\mathbb{K}),
\end{align*}
for $n$ large enough, a contradiction.

If the members of $\mathfrak{B}_2\F$ have no common vertex, then
as $\mathfrak{B}\F$ is intersecting, $\mathfrak{B}_2\F$ must form a triangle. 
As each member of $\mathfrak{B}_2\F$ is a transversal of $\F$, each edge of $\F$ contains at least two vertices in $\mathfrak{B}_2\F$. Let $\F_2$ be the edges of $\F$ that contain exactly two vertices in $\mathfrak{B}_2\F$ and let $\F_3$ be the edges of $\F$ that contain all three vertices in $\mathfrak{B}_2\F$.

By pigeonhole principle, there is a vertex $x$ in $\mathfrak{B}_2\F$ contained in at least $\frac{2}{3}$ of the edges of $\F_2$.
Therefore,
\begin{align*}
d_C(\F) \leq |\F_2|+|\F_3| - C  \frac{2}{3} |\F_2| - C |\F_3| 
= \left(1-\frac{2}{3}C\right)|\F_2| +(1-C)|\F_3|. 
\end{align*}
Observe that $|\F_2| \leq 3\binom{n-3}{r-2}$ and $(1-C)|\F_3| \leq 0$ for $C \geq 1$. In order to satisfy $d_C(\F) \geq d_C(\A_\mathbb{K}) = (3-2C) \binom{n-3}{r-2}$ we must have $|\F_2| = 3\binom{n-3}{r-2}$, i.e., $\mathcal{A}_\mathbb{K} \subseteq  \mathcal{F}$. When $C>1$, we have $\F_3 = \emptyset$ which implies $\mathcal{A}_\mathbb{K} = \mathcal{F}$ and when $C=1$, there is no restriction on $\F_3$ which implies $\mathcal{A}_\mathbb{K} \subseteq  \mathcal{F} \subseteq \mathcal{A}_{\mathbb{K}^+}$. The second part of the theorem follows as any edge in $\F_3$ does not impact the value of $d_C(\F)$ when $C=1$ and $\mathcal{A}_\mathbb{K} \subseteq  \mathcal{F}$.
\end{proof}

We now prepare to prove Theorem~\ref{main-gen}. Very roughly the argument follows that of the proof of Theorem~\ref{32-ker-prop}.
A key step is to connect the $C$-weighted diversity of $\F$ and the $C$-weighted diversity of its flower base and then apply a lemma on the diversity of a ``small'' family constructed from the flower base.

First we need part of the following folklore result. 
This is a special case of the problem to maximize the size of an intersecting $r$-uniform family with minimum transversal of size $r$ (see \cite{zak} for recent progress on the general problem and \cite{FOT} for more on the $r=3$ case).
We include a proof in the Appendix as we could not find a full proof articulated in the literature.

For a family $\B$, by $V(\B)$ we mean the set of vertices spanned by the edges of $\B$, i.e., we ignore potential isolated vertices.

\begin{restatable}[Folklore]{prop}{folkprop}\label{folk}
If $\B$ is an intersecting $3$-uniform family with minimum transversal size $3$, then
\[
|\B| \leq 10 \text{ and } |V(\B)| \leq 7.
\]
\end{restatable}

Let $\B$ be a family and $\varrho : \B \to (0,1]$ be a ``density'' function on the members of $\B$. Define the {\it mass} of the family $\B$ as
\[
\varrho(\B) := \sum_{B \in \B} \varrho(B).
\]
The {\it $\varrho$-degree} of a vertex $x$ is defined as
\[
\varrho(x) := \sum_{\substack{B \in \B \\ x \in B}} \varrho(B),
\]
so the maximum $\varrho$-degree is
\[
\Delta_\varrho(\B) := \max_{x \in V(\B)} \varrho(x).
\]
The notion of diversity generalizes naturally to this setting as
\[
\varrho(\B) - C \cdot  \Delta_\varrho(\B)
\]
for a constant $C$. Note that when $\varrho(B)=1$ for all $B \in \B$ we recover the original definition of $C$-weighted diversity, i.e., $|\B| - C \cdot \Delta(\B) = \varrho(\B) - C \cdot  \Delta_\varrho(\B)$.
This parameter is likely interesting in its own right, but we only use it to establish an intermediate lemma on the way to the proof of Theorem~\ref{main-gen}.

\begin{lemma}\label{baby-lemma}
Fix $\frac{3}{2} \leq C <\frac{7}{3}$.
 Let $\B$ be an intersecting $3$-uniform $7$-vertex family with minimum transversal size $3$. 
 Let $\varrho : \B \to (0,1]$ be a function on the members of $\B$ such that $\varrho(\B) 
 > 6$
 and
 \[
 \varrho(\B) - C \cdot  \Delta_\varrho(\B) > 6 - 3C. 
 \]
 Then $\B$ is the Fano plane (i.e.\ the $(7,3,1)$-design).
\end{lemma}

\begin{proof}
Let $V(\B) = [7]$. Rearranging terms gives

\begin{equation}\label{ineq-chain}
  \Delta_\varrho(\B) < \frac{\varrho(\B)-6}{C}+ 3
   \le \frac{2}{3} \varrho(\B) - 1 
\end{equation}
as $C \geq \frac{3}{2}$ and 
$\varrho(\B) > 6$.

We first show that each pair of vertices in is contained in an edge.

\begin{claim}\label{no-zero-claim}
    $|\B(xy)| > 0$ for all $x,y \in V(\B)$.
\end{claim}

\begin{proof}
      Suppose there is a pair of vertices $1,2$ such that $|\B(12)|=0$. Observe that $\B(1)$ and $\B(2)$ are cross-intersecting, i.e., each member of $\B(1)$ must intersect each member of $\B(2)$ and vice-versa.  
      We now claim that that $\B(1)$ contains two disjoint edges $34$ and $56$. Otherwise, $\B(1)$ is intersecting and therefore is either a star or a triangle. In either case, we can select two vertices of $\B(1)$ that form a transversal of $\B(\overline{1})$ and thus a transversal of $\B$, a contradiction. 
      
\begin{figure}[H]
\begin{center}
{\includegraphics[scale=1.3]{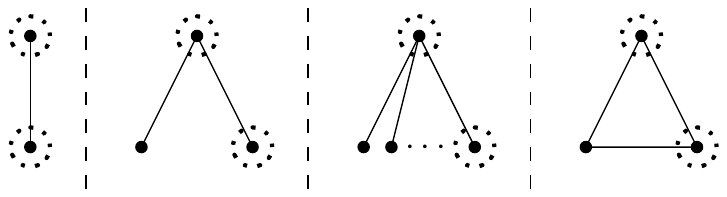}}
\end{center}
\caption{Configurations when $\B(1)$ intersecting with a transversal of $\B$ highlighted.}
\end{figure}

      The same holds for $\B(2)$ and, as $\B(1)$ and $\B(2)$ are cross-intersecting, we have, without loss of generality, $\B(2)$ contains edges $35$ and $46$. 
    Now consider vertex $7$. Observe that no edge of $\B$ contains the pair 17 or 27 as such an edge fails to intersect both of 235 and 236 or both of 134 or 156.
    Therefore, the only possible edges in $\B(7)$ are $36$ and $45$.

    By (\ref{ineq-chain}) we have 
    $\Delta_\varrho(\B) <  \frac{2}{3} \varrho(\B) - 1$. 
    Therefore, the sum of $\varrho$-degrees in $\B$ is
    \begin{align*}
    3 \varrho(\B) & \leq \varrho(1) +  \varrho(2) + \varrho(7) + 4 \cdot \Delta_\varrho(\B) \\
    & <   \varrho(1) +  \varrho(2) + \varrho(7) +  \frac{8}{3} \varrho(\B) - 4 \\
    & \leq 2 + \frac{8}{3} \varrho(\B)
    \end{align*}
     which implies 
     $\varrho(\B) < 6 <\varrho(\B)$, a contradiction.
\end{proof}

Now suppose that $\B$ has a pair of vertices $1,2$ such that $|\B(12)|\geq 2$. Let $123$ and $124$ be two edges. Then, as $\{1,2\}$ is not a transversal, there is an edge $345$. Now observe that there is no edge of $\B$ that contains both $6$ and $7$ as it cannot intersect edges $123,124,345$. This contradicts Claim~\ref{no-zero-claim}. 
Therefore, each pair of vertices in $\B$ is contained in exactly one edge. This implies that $\B$ is the Fano plane.
\end{proof}

We are now ready to prove our main result.

\mainthm*

\begin{proof}[Proof of Theorem~\ref{main-gen}]
Let $\F \subseteq \binom{[n]}{r}$ be an intersecting family with maximum $C$-weighted diversity $d_C(\F)$. Then comparing to $d_C(\A_{\mathbb{F}^+})$ and $d_C(\A_{\mathbb{F}})$ gives $d_C(\F)\geq (7-3C) \binom{n-7}{r-3}$ for all $\frac{3}{2}\leq C < \frac{7}{3}$.

Let $\mathfrak{B}\F$ be the flower base of threshold $r$.
If $\mathfrak{B}\F$ contains a singleton set $\{x\}$, then $d_C(\F) \leq |\F| -|\F(x)| \leq 0$, so we may assume the members of $\mathfrak{B}\F$ are of size at least two.
We first restrict our attention to flower base members of size $2$ or $3$. Define $\B:=\mathfrak{B}_2\F \cup \mathfrak{B}_3\F$. The lower bound and inequality~(\ref{combo-bound}) gives:

\begin{equation*}
(7-3C) \binom{n-7}{r-3} \leq d_C(\F) \leq |\F| \leq \sum_{B \in \B} \binom{n-|B|}{r-|B|} + r_0 \binom{n-4}{r-4}.    
\end{equation*}

This implies that $\B$ is non-empty (for $n$ large enough).
We now distinguish cases based on the size of a minimum transversal of $\B$.

\medskip

{\bf Case 1:} $\tau(\B)=1.$

\medskip

Then every member of $\B$ contains a fixed vertex $x$. Therefore, by (\ref{combo-bound}), $d_C(\F) \leq |\F| - |\F(x)| \leq r_0\binom{n-4}{r-4}$, a contradiction for $n$ large enough.

\medskip

{\bf Case 2:} $\tau(\B)=2.$

\medskip

If $\mathfrak{B}_2\F$ forms a triangle on vertices $\{x,y,z\}$, then as each member of $\mathfrak{B}_2\F$ is a transversal of $\F$, we have that every edge of $\F$ contains at least two vertices from $\{x,y,z\}$, i.e., 
$\F$ has a vertex in $\frac{2}{3}$ of the edges. Thus, as $C\geq \frac{3}{2}$, Lemma~\ref{ker-lemma} implies $d_C(\F) \leq 0$, a contradiction.

Now suppose that $\mathfrak{B}_2\F$ does not form a triangle. Then  $\mathfrak{B}_2\F$ forms a star with center $x$, or $\mathfrak{B}_2\F$ is empty. If $\mathfrak{B}_2\F$ is nonempty,  pick any $xy \in \mathfrak{B}_2\F$ and note that $\{x,y\}$ is a transversal of $\B$. If $\mathfrak{B}_2\F$ is empty, let $\{x,y\}$ be any transversal of $\B$. In both cases $\{x,y\}$ is also transversal of $\F$.
Now, see that $|\F(x)| \leq \Delta(\F)$. Moreover, as $y$ is in each edge not containing $x$, we have $|\F(\overline{x})| \leq |\F(y)| \leq \Delta(\F)$.
Therefore we have the following two bounds:
\[
(7-3C) \binom{n-7}{r-3} \leq d_C(\F) \le 
|\F(x)| + |\F(\overline{x})| - C  |\F(x)| = (1-C)|\F(x)| + |\F(\overline{x})|
\]
and
\[
(7-3C) \binom{n-7}{r-3} \leq d_C(\F) \leq  |\F(x)| + |\F(\overline{x})| - C  |\F(\overline{x})| = |\F(x)| + (1-C)|\F(\overline{x})|.
\]
Multiplying the second inequality by $(C-1)$ and adding the first inequality gives
\[
C(7-3C) \binom{n-7}{r-3} \leq ((C-1)(1-C)-1) |\F(\overline{x})| = (2C-C^2) |\F(\overline{x})|.
\]
Let $\ell$ be the number of members of $\B$ that do not contain $x$, and observe each is of size at least $3$ as all members of $\mathfrak{B}_2\F$ contain $x$. Therefore, 
\[
(7-3C) \binom{n-7}{r-3} \leq (2-C) |\F(\overline{x})| \leq  (2-C) \ell \binom{n-3}{r-3}
\]
which, for $n$ large enough, implies $\ell=|\B(\overline x)|>4$.

The set $\{x,y\}$ is a minimum transversal of $\B$ which implies there is a member of $\B$ containing $x$ and not $y$. If there is a member of size two distinct from $xy$, i.e., $xz \in \B$, then all members of $\B(\overline{x})$ must contain $z$ as $\B$ is intersecting. As $|\B(\overline x)|>4$, in order to maintain the intersecting property, all members of $\B(\overline{y})$ must also contain $z$. Therefore, all members of $\B$ contain at least two of $\{x,y,z\}$ which implies that there is a vertex in $\frac{2}{3} \geq \frac{1}{C}$ of the edges of $\F$. Again Lemma~\ref{ker-lemma} implies that $d_C(\F) \leq 0$, a contradiction.

 So now the only potential member of $\mathfrak{B}_2\F$ is $xy$. 
 Whether $xy$ is a member of $\mathfrak{B}_2\F$ or not, we can repeat the estimates on $d_C(\F)$ above with $y$ in place of $x$ to obtain $|\B(\overline{y})|> 4$ as well. Let $xzw\in \B(\overline{y})$. By pigeonhole, and without loss of generality, at least three members of $\B(\overline{x})$ contain $z$. No member of $\B(\overline{y})$ can intersect each of these edges without containing $z$. It readily follows that any member of $\B$ must contain two of $\{x,y,z\}$, so again $d_C(\F) \leq 0$, a contradiction.

\medskip

{\bf Case 3:} $\tau(\B) = 3.$

\medskip

Each member of the flower base is a transversal of $\F$, so $\B$ contains no member of size $2$. Therefore, $\B$ is an intersecting $3$-uniform family with minimum transversal size $3$. Proposition~\ref{folk} implies $|V(\B)|\leq 7$. 

Suppose $|V(\B)|\leq 6$. Let $\F_3$ be the edges of $\F$ that intersect $V(\B)$ in exactly three vertices and $\F_{\geq 4}$ be edges of $\F$ that intersect $V(\B)$ in at least four vertices. Suppose there is an edge $F$ of $\F$ that intersects $V(\B)$ in at most two vertices. As each $B \in \B$ is a transversal of $\F$, the vertices of $F \cap V(\B)$ form a transversal of size at most $2$ of $\B$, a contradiction. Therefore, $\F = \F_3 \sqcup \F_{\geq 4}$.

For each subset $T \subset V(\B)$ of size $3$ there is a corresponding complement $T' =  V(\B) \setminus T$. Observe that the link graphs $\F(T)$ and $\F(T')$ are cross-intersecting and do not contain elements of $T\cup T'$. Therefore, by (\ref{cross-ineq}) we have
\[
|\F(T)| + |\F(T')| \leq \binom{n-6}{r-3}.
\]
Summing over all $10 = \frac{1}{2} \binom{6}{3}$ pairs $\{T,T'\}$ we have
\[
|\F_3| \leq 10 \binom{n-6}{r-3}.
\]
Now, for $C\geq \frac{3}{2}$ we can estimate the diversity as:
\begin{align}\label{632-ub}
d_{C}(\F) \leq |\F_3| + |\F_{\geq 4}| - C\left(\frac{3}{6} |\F_3| + \frac{4}{6}|\F_{\geq 4}|\right) \le \left(1-C\frac{1}{2}\right) |\F_3|   \leq (10-5C) \binom{n-6}{r-3}.
  \end{align}

For $C> \frac{3}{2}$ we have $7-3C > 10-5C$. Therefore,
\[
d_C(\F) \geq  (7-3C) \binom{n-7}{r-3} > (10-5C) \binom{n-6}{r-3},
\]
for $n$ large enough, contradicting (\ref{632-ub}).
When $C=\frac{3}{2}$, we have $7-3C = 10-5C = \frac{5}{2}$. Therefore,
\[
d_{3/2}(\F) \geq d_{3/2}(\A_{\mathbb{F}^+}) = \frac{5}{2} \binom{n-7}{r-3} + 4 \binom{n-7}{r-4}  >\frac{5}{2} \binom{n-6}{r-3}
\]
which again contradicts (\ref{632-ub}).

So now we may assume $|V(\B)|=7$.
Let $\varrho(B)$ be the proportion of total possible edges of $\F$ intersecting $V(\B)$ in exactly the $3$-set $B$, i.e.,
\[
\varrho(B) := \frac{|\F(B)|}{\binom{n-7}{r-3}}\in (0,1].
\]
The average $\varrho$-degree of a vertex in $V(\B)$ is $\frac{3}{7} \varrho(\B)$. Therefore,
by (\ref{combo-bound}),
\[
(7-3C)\binom{n-7}{r-3} \leq d_C(\F) \leq \varrho(\B) \binom{n-7}{r-3} - C \frac{3}{7} \varrho(\B)\binom{n-7}{r-3} + r_0\binom{n-4}{r-4}
\]
which implies $\varrho(\B) \geq 7 - o(1) > 6$ for $n$ large enough.

Now, edges of $\F$ all fit into at least one of three types: those that intersect $V(\B)$ in exactly a member of $\B$, those that intersect $V(\B)$ in four or more vertices, and those that contain a member of $\mathfrak{B}\F$ of size at least $4$. As there are no more than $2^7$ subsets of $V(\B)$, $|\F|$ can be estimated as:
\[
|\F| \leq \sum_{B \in \B} \varrho(B) \binom{n-7}{r-3} + 2^7 \binom{n-7}{r-4} + r_0\binom{n-4}{r-4} \le \varrho(\B)\binom{n-7}{r-3}  + 3r_0 \binom{n-4}{r-4}.
\]
We can estimate the maximum degree in $\F$ as:
\[
 \Delta(\F) \geq   \max_{x \in V(\B)} \left \{  \sum_{\substack{B \in \B \\ x \in B}} \varrho(B) \binom{n-7}{r-3} \right\} =   \Delta_\varrho(\B)\binom{n-7}{r-3}.
\]
Therefore,
 \begin{align*}
d_C(\F) \leq  \left( \varrho(\B) - C \cdot \Delta_\varrho(\B) \right) \binom{n-7}{r-3} + 3r_0 \binom{n-4}{r-4}.
 \end{align*}
 
Now, as $d_C(\F) \geq (7-3C)\binom{n-7}{r-3}$, we have $\varrho(\B) - C \cdot \Delta_\varrho(\B) \geq 7-3C - o(1) > 6-3C$ for $n$ large enough. By Lemma~\ref{baby-lemma} we have that $\B$ is the Fano plane.

Given that $\B$ is the Fano plane, one can easily see that $\B$ must be the entire flower base of $\F$. Any member $B \in \mathfrak{B}\F$ of the flower base intersects each member of $\B$, so $B$ is a transversal of $\B$ and intersects $V(\B)$ in at least three vertices. The transversals of size $3$ or $4$ of the Fano plane either contain or are members of the Fano plane, and any set of $5$ or more vertices of $V(\B)$ contains a member of the Fano plane. Therefore, by the Sperner condition of $\mathfrak{B}$, we have $B\in \B$.

Let $\F_3$ be the edges of $\F$ that intersect $V(\B)$ in exactly a member of $\B$. Let $\F_4$ be the edges $\F$ that intersect $V(\B)$ in exactly $4$ vertices and $\F_{\geq 5}$ be the edges $\F$ that intersect $V(\B)$ in at least $5$ vertices.
 Since each edge of $\F$ must contain a member of $\B$ we have $\F = \F_3 \sqcup \F_4 \sqcup \F_{\geq 5}$.
The sum of degrees of vertices in $V(\B)$ is 
$\sum_{v\in V(\B)} |F(v)| \geq  3|\F_3|+4|\F_4|+5|\F_{\geq 5}|$. The maximum degree of a vertex in $\F$ is at least the average degree of a vertex in $V(\B)$, so
\[
\Delta(\F) \geq \frac{3}{7}|\F_3| +\frac{4}{7}|\F_4| +\frac{5}{7}|\F_{\geq 5}|.
\]

We can bound the diversity $d_C(\F)$ as follows.
\begin{align*}
     d_C(\F) & = |\F| - C\cdot \Delta (\F)\\
    & \leq  |\F_3|+|\F_4|+|\F_{\geq 5}| - C\left(\frac{3}{7}|\F_3| +\frac{4}{7}|\F_4| +\frac{5}{7}|\F_{\geq 5}|\right)\\
    &= \left(1-C \frac{3}{7}\right)|\F_3| +\left(1-C \frac{4}{7}\right)|\F_4| +\left(1-C \frac{5}{7}\right)|\F_{\geq 5}|.   
\end{align*}
As $\B$ is the Fano plane it has $7$ members, so we have $|\F_3| \leq  7\binom{n-7}{r-3}$.
For $C > \frac{7}{4}$, the coefficients of $|\F_4|$ and $|\F_{\geq 5}|$ are negative. So, in order to satisfy
\[
(7-3C) \binom{n-7}{r-3} = d_C(\A_{\mathbb{F}}) \leq d_C(\F), 
\]
we must have $|\F_3| =  7\binom{n-7}{r-3}$ and $\F_4$ and $\F_{\geq 5}$ both empty, i.e., $\F$ is isomorphic to the family $\A_{\mathbb{F}}$.

For $\frac{3}{2} \leq C \leq \frac{7}{4}$, the quantity $\left(1-C \frac{5}{7}\right)$ is negative. No edge of $\F$ intersects $V(\B)$ in exactly $4$ vertices that are disjoint from a $B \in \B$. Therefore, $|\F_4| \leq 28 \binom{n-7}{r-4}$.
As before, in order to satisfy
\[
(7-3C) \binom{n-7}{r-3} +  (28-16C) \binom{n-7}{r-4} = d_C(\A_{\mathbb{F}^+}) \leq d_C(\F)
\]
we must have $|\F_3| = 7\binom{n-7}{r-3}$ and $\F_{\geq 5}$ is empty.
When $C< \frac{7}{4}$ we must have $|\F_4| =  28\binom{n-7}{r-4}$, i.e., $\F$ is isomorphic to the family $\A_{\mathbb{F}^+}$. 
When $C=\frac{7}{4}$, this implies that $\F$ contains $\A_\mathbb{F}$ and is contained in $\A_{\mathbb{F}^+}$. Since $|\F_3| = 7 \binom{n-7}{r-3}$, each vertex of $V(\B)$ is contained in the same number of edges of $|\F_3|$. This implies that each vertex of $V(\B)$ must also be contained in the same number of edges of $\F_4$, i.e., $d_C(\F \setminus \A_\mathbb{F})=0$.
\end{proof}

We now prove Theorem~\ref{main-gen-C}.

\thmthree*

\begin{proof}
We begin with the upper bound. First assume $\F$ has a transversal of size at most $q$. Then some vertex $x$ in that transversal sees at least $\frac{1}{q}$ proportion of the edges of $\F$. 
Observe that $C \frac{1}{q} \geq \frac{q^2+q+1}{q^2+q}>1$, so
\begin{align*}
    d_C(\F) \leq |\F| -C  |\F(x)| 
     &\leq \left(1-C\frac{1}{q}\right)|\F|
    < 0
\end{align*}
which establishes the upper bound.

Thus, we can suppose $\F$ satisfies $\tau(\F)>q$, so the members of its flower base (with threshold $r$) are of size at least $q+1$. Then by (\ref{combo-bound}), we get the following bound on the diversity of $\F$.
\begin{align*}
    d_C(\F) \leq |F|
     = \sum_{B\in \mathfrak{B}\F}\binom{n-|B|}{r-|B|}
    \leq r_0 \binom{n-q-1}{r-q-1}
    =O(n^{r-q-1}).
\end{align*}

For the lower bound, let $n$ be sufficiently large and let $\mathcal{P}$ be a finite projective plane of order $p$ (i.e., a
$(p^2+p+1,p+1,1)$-design) on vertex set $[p^2+p+1]$. These are known to exist for all prime powers $p$. We define the family
\[
\F := \{ F\in\binom{[n]}{r} : F\cap [p^2+p+1]\in \mathcal{P}\}.
\]
Since $\mathcal{P}$ is intersecting, $\F$ is intersecting.
It is straightforward to see that $\F$ has $C$-weighted diversity given by
\begin{align*}
    d_C(\F) &= |\F|-C\cdot\Delta(F)\\
    &= (p^2+p+1)\binom{n-p^2-p-1}{r-p-1}-C\cdot \Delta(\mathcal{P})\binom{n-p^2-p-1}{r-p-1}\\
    &= (p^2+p+1-C\cdot \Delta(\mathcal{P}))\binom{n-p^2-p-1}{r-p-1}\\
    &= \left(1 - C\cdot \frac{p+1}{p^2+p+1}\right)(p^2+p+1)\binom{n-p^2-p-1}{r-p-1}\\
    & = \Omega(n^{r-p-1}). \qedhere
\end{align*}
\end{proof}

\section{Stability}\label{stability-section}

In this section we prove two stability theorems for diversity. The first shows that a ``two out of three'' family $\A$ (i.e.\ $\A_\mathbb{K} \subseteq \A \subseteq \A_{\mathbb{K}^+}$) is stable
and the second is a short reproof (for a much worse threshold on $n$) of a theorem due to Kupavskii and Zakharov~\cite{KuZa} (see also Frankl~\cite{Fr-ekr-deg}).

\begin{thm}
Let $r\geq 3$ and fix $0 \leq t \leq \varepsilon \binom{n-3}{r-2}$ for $0<\varepsilon<1$.
If $\F \subseteq \binom{[n]}{r}$ is an intersecting family with diversity
\[
d(\F)= \binom{n-3}{r-2}-t,
\]
then we can add at most $3t$ edges to $\F$ to
obtain a ``two out of three'' family, for $n$ large enough.
\end{thm}

\begin{proof}
Let  $\mathfrak{B}\F$ be the flower base with threshold $r$.
If $\mathfrak{B}\F$ contains a singleton set $\{x\}$, then by Lemma~\ref{inherit}, every edge of $\F$ contains vertex $x$, i.e., $\F$ has diversity $0$. 

If $\mathfrak{B}\F$ has no member of size $2$, then (\ref{combo-bound}) gives
\[
d(\F) <|\F| \leq r_0 \binom{n-3}{r-3} <  (1-\varepsilon) \binom{n-3}{r-2} \leq \binom{n-3}{r-2}-t,
\]
for $n$ large enough, a contradiction.
Therefore, $\mathfrak{B}\F$ contains a member of size $2$, i.e., $\mathfrak{B}_2\F \neq \emptyset$. Suppose that the members of size $2$ have a common vertex $x$. Then by (\ref{combo-bound}),
\begin{align*}
d(\F) & \leq |\F| -   |\F(x)| 
\leq |\F(x)| +  r_0 \binom{n-3}{r-3} -  |\F(x)|  < (1-\varepsilon) \binom{n-3}{r-2} \leq \binom{n-3}{r-2}-t,
\end{align*}
for $n$ large enough, a contradiction. 

As $\mathfrak{B}\F$ is intersecting, $\mathfrak{B}_2\F$ is then a triangle and spans vertices $\{x,y,z\}$. Moreover, as each member of $\mathfrak{B}_2\F$ is a transversal of $\F$, each edge of $\F$ must intersect $\{x,y,z\}$ in at least two vertices. Therefore, $\F$ is a subfamily of $\A_{\mathbb{K}^+}$
and we need to simply add edges to $\F$ to get a ``two out of three'' family. It remains to compute how many edges we need to add.

Consider the ``two out of three family'' $\A$ that consists of $\A_{\mathbb{K}}$ and all edges of $\F$ containing $\{x,y,z\}$. The family $\A$ has $3\binom{n-3}{r-2} + |\F(xyz)|$ edges and contains $\F$.
The number of edges of $\F$ containing $x$ is $|\F(xy)|+|\F(xz)|-|\F(xyz)| \leq \Delta(\F)$ and
$|\F| = |\F(xy)|+|\F(xz)|+|\F(yz)|- 2 |\F(xyz)|$. Therefore,
\begin{align*}
 \binom{n-3}{r-2}-t = |\F| - \Delta(\F) \leq  |\F(yz)|-|\F(xyz)|.
\end{align*}

The same argument gives $\binom{n-3}{r-2}-t \leq |\F(xy)|-|\F(xyz)|$ and $\binom{n-3}{r-2}-t \leq |\F(xz)| -|\F(xyz)|$. Thus,
\[
3 \binom{n-3}{r-2}-3t \leq |\F(xy)|+|\F(xz)|+|\F(yz)| - 3|\F(xyz)|  = |\F| - |\F(xyz)|.
\]
Therefore, 
$|\F| \geq 3\binom{n-3}{r-2} + |\F(xyz)| -  3t$, i.e., we need to add at most $3t$ edges to $\F$ to obtain the ``two out of three'' family $\A$.
\end{proof}

\begin{thm}[Kupavskii, Zakharov \cite{KuZa}]
Let $\F \subseteq \binom{[n]}{r}$ be an intersecting family.
If 
\[
|\F| > \binom{n-1}{r-1}  - \binom{n-u-1}{r-1} + \binom{n-u-1}{r-u}
\]
for a real $3 \leq u \leq r$, then
\[
d(\F) < \binom{n-u-1}{r-u},
\]
for $n$ large enough.
\end{thm}

\begin{proof}
Let $\mathfrak{B}\F$ be the flower base with threshold $r$.
First note that $\tau(\F) \geq 2$ as otherwise $d(\F)=0$. 
Therefore, we may assume $\mathfrak{B}_1\F$ is empty.
First observe that
\begin{align*}
    |\F|  > \binom{n-1}{r-1}  - \binom{n-u-1}{r-1} + \binom{n-u-1}{r-u} \geq u\binom{n-u-1}{r-2}.
\end{align*}
Therefore, by (\ref{combo-bound}), we have
\begin{align*}
 u\binom{n-u-1}{r-2}  < |\F| \leq 
 |\mathfrak{B}_2\F| \binom{n-2}{r-2} + r_0\binom{n-3}{r-3}.
\end{align*}
This implies that $\mathfrak{B}_2\F$ contains at least $u$ members.
Now, as the flower base is intersecting (by Lemma~\ref{inherit}), this means that $\mathfrak{B}_2\F$ must either be a star with center $x$ or a triangle.

Let us first address the case when the members of $\mathfrak{B}_2\F$ form a triangle. 
Then $\F$ is a subfamily the ``two out of three'' family $\A_{\mathbb{K}^+}$. Therefore,
\[
|\F| \leq |\A_{\mathbb{K}^+}| = 3 \binom{n-3}{r-2} +\binom{n-3}{r-3}.
\]
On the other hand, in this case we must have $u=3$, so
\begin{align*}
|\F| & > \binom{n-1}{r-1} - \binom{n-4}{r-1} + \binom{n-4}{r-3} \\
& = 3 \binom{n-4}{r-2} + 3\binom{n-4}{r-3} + \binom{n-4}{r-4} + \binom{n-4}{r-3} \\
&= 3 \binom{n-3}{r-2} + \binom{n-3}{r-3},
\end{align*}
a contradiction.

Now we may suppose that the members of $\mathfrak{B}_2\F$ form a star.
If every member of $\mathfrak{B}\F$ contains $x$, then $d(\F)=0$ and we are done. So suppose that $\mathfrak{B}\F$ contains at least one member $Y$ of size greater than $2$ that does not contain $x$. As the flower base is intersecting, $Y$ must contain the union $\bigcup_{B\in\mathfrak{B}_2\F}B\setminus\{x\}$. Thus, $|Y| \geq u$. 
Each edge of $\F$ must contain $x$ or $Y$. Moreover, as $Y \in \mathfrak{B}\F$, the petals $\F(Y)$ have transversal greater than $r$. Thus, each edge containing $x$ must also intersect $Y$. Therefore, the number of edges of $\F$ satisfies
\[
|\F| \leq \binom{n-1}{r-1} - \binom{n-|Y|-1}{r-1} + \binom{n-|Y|-1}{r-|Y|}.
\]
Given the strict lower bound on $|\F|$, this implies $|Y| > u$.
Now,
\[
d(\F)  \leq |\F| - |\F(x)| =  |\F(Y)| - |\F(Y\cup\{x\})| 
 \leq \binom{n-|Y|-1}{r-|Y|} < \binom{n-u-1}{r-u}. \qedhere
\]
\end{proof}

\section{Concluding remarks}\label{final-section}

In this paper we introduced a generalization of diversity of an intersecting family $\F \subseteq \binom{[n]}{r}$ as
\[
d_C(\F) : = |\F| - C \cdot \Delta (\F).
\]
The original version of diversity is the case when $C=1$. For $0 \leq  C < \frac{7}{3}$ we determined the maximum possible diversity $d_C(\F)$ and characterized the families achieving this maximum. For the range $1 < C < \frac{3}{2}$, these families are ``two out of three'' families and for $\frac{3}{2} \leq C < \frac{7}{3}$ they are the ``Fano-based'' families containing $\mathcal{A}_\mathbb{F}$. We strongly believe that, for the range $\frac{7}{3} \leq C < \frac{13}{4}$, the maximal constructions will be based on the finite projective plane of order $3$ and have diversity $d_C(\F)$ approximately
\[
(13-4C) \binom{n-13}{r-4}.
\]
As $C$ gets larger, constructions become less clear as an appropriate finite projective plane may not exist. However, our understanding of the flower base allows us to observe that $d_C(\F)=O(n^{r-4})$ for $C\geq \frac{7}{3}$ To see this, recall that $d_C(\F)\leq |\F|$. So, if no member of the flower base exists of size at most $3$ for $n$ large, the bound follows. Our approach from Theorem~\ref{main-gen} can be used to show that no transversals of size one or two exist of the flower base of families with positive diversity for $C$ in this range. Should there exist transversals of size $3$ in the flower base, we may have nonempty $\mathfrak{B}_3\F$. However, the transversal condition still holds and shows that members of  $\mathfrak{B}_3\F$ are on a vertex set $V$ of size at most $7$. One can then show that a vertex in $V$ is contained in at least $\frac{3}{7}$ of the members of $\F$, which then gives that the diversity of the family is non-positive. Thus, families with positive $C$-weighted diversity have $O(n^{r-4})$ edges.
We note that this bound is an improvement on that suggested by Theorem~\ref{main-gen-C}, which only gives that $d_C(\F)=O(n^{r-3})$ for $C=\frac{7}{3}$.

There several other generalizations of the diversity question or measure of ``distance from a star.'' Frankl and Wang \cite{frankl2022best} define the {\it $\ell$-diversity} of a family $\F \subseteq \binom{[n]}{r}$ as the maximum size of the unlink $|F(\overline{S})|$ over all $\ell$-sets $S$ of vertices. As the ordinary diversity can be realized as $d(\F) = \max_{x \in [n]} |\F(\overline{x})|$, this is a natural generalization.

Interestingly, the extremal family for the double diversity (i.e.\ the case when $\ell=2$) are very similar to the Fano-based constructions $\A_{\mathbb{F}^+}$ and $\A_{\mathbb{F}}$ for $d_C(\F)$. It's not immediately clear how to connect the double diversity and $d_C(\F)$, however. They also believe that a construction based on the finite projective plane of order $3$ will maximize $3$-diversity.

In \cite{frankl2023improved}, Frankl and Wang discuss the problem to maximize $\frac{\Delta(\F)}{|\F|}$
for an intersecting family $\F \subseteq \binom{[n]}{r}$. They conjecture that for $n > 100r$, if $|\F| > \binom{n-3}{r-3}$, then
\[
\frac{\Delta(\F)}{|\F|} \geq \frac{3}{7}.
\]
We note that Theorem~\ref{main-gen} implies an answer to this conjecture for large $n$. Indeed, for $2 \leq C = \frac{7}{3}-\varepsilon$, we have
\[
|\F| - C \cdot \Delta(\F) \leq (7-3C)\binom{n-3}{r-3}.
\]
Rearranging terms and using $|\F| > \binom{n-3}{r-3}$ and  $C = \frac{7}{3} - \varepsilon$, gives
\[
 \frac{\Delta(\F)}{|\F|} \geq \frac{3C-6}{C} \geq \frac{3-9\varepsilon}{7-3\varepsilon}.
\]
As $\varepsilon$ can be chosen arbitrarily small, we  have $\frac{\Delta(\F)}{|\F|} \geq \frac{3}{7}$.




\bibliography{master}{}
\bibliographystyle{hplain}

\section{Appendix}\label{appendix-section}

We give a proof of the folklore result from Section~\ref{main-section}.

\folkprop*

\begin{proof}
Fix an edge $123 \in \B$ and note that it is a transversal. We shall denote the members of $\B(1)$ containing neither $2$ nor $3$ by $\B_1$. Note that the members of $\B_1$ are of size $2$ and $\B_1$ is non-empty as otherwise $\{2,3\}$ is a transversal of size $2$. We similarly define $\B_2$ and $\B_3$. These three families are pairwise cross-intersecting. Define $\B_{12}$ as the members of $\B(12)$ not containing $3$. Similarly define $\B_{13}$ and $\B_{23}$ and note that the members of these families are all of size $1$. 
We will use that $\B_x$ and $\B_{yz}$ are cross-intersecting and therefore $|\B_{yz}| \leq 2$ when $\B_x$ is a single edge and $|\B_{yz}|\leq 1$ when $\B_x$ contains two intersecting edges and $|\B_{yz}|=0$ when $\B_x$ contains two disjoint edges.

\begin{figure}[H]
\begin{center}
{\includegraphics[scale=1.3]{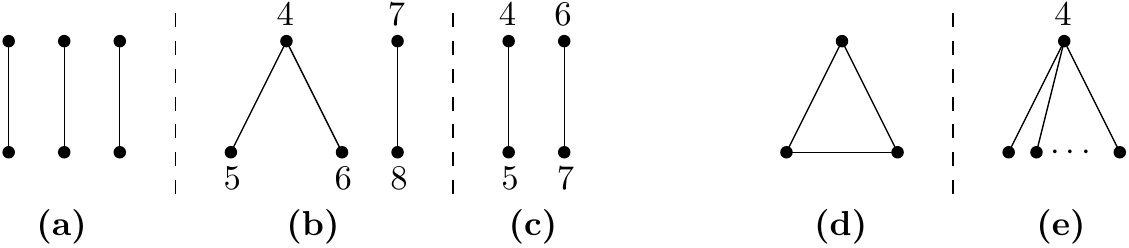}}
\end{center}
\caption{Configurations of $\B_1$ in {\bf Cases (a), (b), (c)} and $\B'$ in {\bf Cases (d), (e)}.}
\end{figure}

We begin by investigating the cases when $\B_1$ is not intersecting.

\medskip
{\bf Case (a):} $\B_1$ contains $3$ disjoint edges.
\medskip

No edge of $\B_2$ can intersect $3$ disjoint edges, so $\B_2$ is empty, a contradiction.

\medskip
{\bf Case (b):} $\B_1$ contains edges $45, 46$ and  $78$.
\medskip

As $\B$ is intersecting, the edges of $\B_2$ and $\B_3$ must all contain vertex $4$ and $\B_{23}$ is empty. Thus, every edge of $\B$ contains $1$ or $4$, i.e., there is a transversal of size $2$, a contradiction.

\medskip
{\bf Case (c):} $\B_1$ consists of $2$ disjoint edges $45$ and $67$.
\medskip

By the previous two cases these two edges span the entire vertex set of $\B_1$. As $\B$ is intersecting, $\B_2$ and $\B_3$ are contained in the vertex set $\{4,5,6,7\}$. The same holds for $\B_{13}$ and $\B_{12}$ and $\B_{23}$ is empty, so $|V(\B)|=7$.

If there is a vertex $a$ in all edges of $\B_2$ and $\B_3$, then $\{1,a\}$ is a transversal of size $2$ of $\B$, a contradiction. 
From here it is easy to see (without loss of generality) that
 $\B_2$ contains disjoint edges $46$ and $57$ and, since $\B_3$ is non-empty, we may assume it contains edge $56$. In order to maintain the intersecting property of $\B$ the only remaining potential edges are $56 \in \B_1$, $56 \in \B_2$, and either $5,6 \in \B_{12}$ or $47 \in \B_3$. Both $\B_{23}$ and $\B_{13}$ are empty. In any case, this gives  $|\B| \leq 10$.

\medskip

The above cases also hold for $\B_2$ and $\B_3$, so
we may now suppose that $\B_1,\B_2,\B_3$ are each intersecting. Let $\B'$ be the multigraph consisting of the edges of $\B_1$, $\B_2$, $\B_3$. As $\B$ is intersecting, $\B_1,\B_2,\B_3$ are pairwise cross-intersecting. Therefore, $\B'$ is intersecting.

\medskip
{\bf Case (d):} $\B'$ is a triangle.
\medskip

We observe $\B_{12}$ cannot be disjoint from $V(\B_3)$. A similar argument for $\B_{13}$ and $\B_{23}$ is immediate. Hence, $V(\B)$ consists of $1,2,3$ and the vertices of the triangle and so $|V(\B)|=6$. As $\B$ is intersecting there is no edge and its complement in $\B$. Therefore, $|\B| \leq \frac{1}{2} \binom{6}{3} = 10$. 

\medskip
{\bf Case (e):} $\B'$ is a star with center vertex $4$.
\medskip

First note that if $\B_1$ contains two edges, then the only possible member of $\B_{23}$ is $4$. As the edges of $\B_2$ and $\B_3$ also contain $4$, we have that $1$ and $4$ is a transversal of size $2$ of $\B$, a contradiction. Therefore, $|V(\B')|\le 4$ and thus $|V(\B)|\leq 7$. Moreover, as each of $\B_1,\B_2,\B_3$ contains a single edge, there are at most $6$ edges among $\B_{23},\B_{13},\B_{12}$ for a total of at most $10$ edges in $\B$.
\end{proof}

\end{document}